\documentclass[10pt,reqno]{amsart}

\usepackage[a4paper]{geometry}
\usepackage{enumerate,enumitem}
\usepackage[utf8]{inputenc}
\usepackage[T1]{fontenc} 
\usepackage[english]{babel}
\usepackage{array}
\usepackage{amsmath,amssymb,amsthm}
\usepackage{xfrac}
\usepackage{xcolor}
\usepackage{graphicx}
\usepackage{mathrsfs}
\usepackage{url}
\usepackage{hyperref}
\hypersetup{colorlinks=true,breaklinks=true,urlcolor=blue,linkcolor=blue,bookmarksopen=true} 
\usepackage{enumitem}

\newcommand{\be} {\begin{equation}}
\newcommand{\ee} {\end{equation}}
\newcommand{\bea} {\begin{eqnarray}}
\newcommand{\eea} {\end{eqnarray}}
\newcommand{\Bea} {\begin{eqnarray*}}
\newcommand{\Eea} {\end{eqnarray*}}

\newcommand*{\X}{\mathcal{X}}
\newcommand*{\M}{\mathcal{M}}

\renewcommand*{\L}{\mathcal L}

\newcommand{\1}{\mathbf 1}
\newcommand*{\N}{\mathbb{N}}
\newcommand*{\B}{\mathcal{B}}

\newcommand*{\R}{\mathbb{R}}

\newcommand*{\e}{\mathrm{e}}

\newcommand{\dis}{\displaystyle}

\newtheorem{theo}{Theorem}[section]

\newtheorem{lem}[theo]{Lemma}

\newtheorem{hyp}{Hypothesis}

\numberwithin{equation}{section}

\begin{document}

\title[On irreducibility and ergodicity of nonconservative semigroups]{On an irreducibility type condition for the ergodicity of nonconservative semigroups}

\author{
	Bertrand \textsc{Cloez}
		\and
	Pierre \textsc{Gabriel}
}
\def\runauthor{
	Bertrand \textsc{Cloez} \and Pierre \textsc{Gabriel}
}

\date{}
    \address[B. \textsc{Cloez}]{MISTEA, INRAE, Montpellier SupAgro, Univ. Montpellier, 2 place Pierre Viala, 34060 Montpellier, France.}
    \email{bertrand.cloez@inrae.fr}
    \address[P. \textsc{Gabriel}]{Laboratoire de Math\'ematiques de Versailles, UVSQ, CNRS, Universit\'e Paris-Saclay,  45 Avenue des \'Etats-Unis, 78035 Versailles cedex, France.}
    \email{pierre.gabriel@uvsq.fr}

\begin{abstract}
We propose a simple criterion, inspired from the irreducible aperiodic Markov chains, to derive the exponential convergence of general positive semigroups.
When not checkable on the whole state space, it can be combined to the use of Lyapunov functions.
It differs from the usual generalization of irreducibility and is based on the accessibility of the trajectories of the underlying dynamics.
It allows to obtain new existence results of principal eigenelements, and their exponential attractiveness, for a nonlocal selection-mutation population dynamics model defined in a space-time varying environment.
\end{abstract}



\maketitle

\section{The main result}

We are interested in the long time behavior of positive semigroups in \emph{weighted signed measures} spaces.
Before dealing with these semigroups, let us start by defining more precisely what we mean by weighted signed measures.

\

Let $\X$ be a measurable space.
We work in weighted signed measures spaces on $\X$.
More precisely for a weight function $V:\X\to(0,\infty)$ we denote by $\M_+(V)$ the set of positive measures on $\X$ which integrate $V$
and we define the space of weighted signed measures as
\[\M(V)=\M_+(V)-\M_+(V)\]
the smallest vector space with positive cone $\M_+(V)$, see~\cite{BCGM} for a rigorous construction as a quotient space.
Basically, an element $\mu$ of $\M(V)$ is the difference $\mu=\mu_+-\mu_-$ of two positive measures $\mu_+,\mu_-\in\M_+(V)$ which are mutually singular (Hahn-Jordan decomposition).
It acts on the Banach space
\[\B(V)=\Big\{f:\X\to\R\ \text{measurable},\ \left\|f\right\|_{\B(V)}:=\sup_{x\in\X} |f(x)| /V(x)<+\infty\Big\}\]
through
\[\mu(f)=\int_\X f\,d\mu_+-\int_\X f\,d\mu_-.\]
The space $\M(V)$ is a Banach space for the weighted total variation norm
\[\left\|\mu\right\|_{\M(V)}=\sup_{\|f\|_{\B(V)}\leq1}\mu(f).\]
Notice that, with these definitions, the standard total variation norm is $\left\|\cdot\right\|_{\rm{TV}}=\left\|\cdot\right\|_{\M(\1)}$, where $\1$ stands for the mapping $x\mapsto 1$.

\

We consider positive semigroups $(M_t)_{t\geq0}$ acting on $\M(V)$ on the left ($\mu\mapsto\mu M_t$) and on $\B(V)$ on the right ($f\mapsto M_tf$),
which enjoy the classical duality relation $\mu(M_tf)=(\mu M_t)(f)$.
We present sufficient irreducibility type conditions which, possibly combined to Lyapunov type conditions, ensure the so-called \emph{asynchronous exponential behavior} of the semigroup:
\[\mu M_t\sim \e^{\lambda t}\mu(h)\gamma\qquad\text{as}\ t\to\infty,\]
where $\lambda\in\R$, $h$ is a positive function, and $\gamma$ is a positive measure.

\

We start with the global case, {\it i.e.} conditions verified on the whole state space $\X$.

\begin{hyp}
\label{ass:global}
The weight function is given by $V=\1$.
There exist $\tau>0$ and $C\geq1$ such that $C^{-1}\leq M_s\1\leq C$ for all $s\in[0,\tau]$,
and there exist $c>0$ and a family of probability measures $(\sigma_{x,y})_{x,y\in \X}$ over $[0,\tau]$ such that
\begin{enumerate}[label=(H\arabic*), parsep=4mm]
\item \label{H:irr}
\makebox[.83\textwidth][c]{
$\dis\forall x,y \in \X, \qquad \delta_x M_{\tau}(\cdot) \geq c \int_0^\tau \delta_y M_{\tau-s} (\cdot) \,\sigma_{x,y}(ds)$,
}
\item \label{H:aper}
\makebox[.83\textwidth][c]{
$\dis\sup_{x,x'\in \X} \inf_{y \in \X} \left\| \sigma_{x,y} - \sigma_{x',y} \right\|_{\rm{TV}} <2$.
}
\end{enumerate}
\end{hyp}

\

Let us do some comments on these two assumptions.

\smallskip

Condition~\ref{H:irr} ensures the existence of a time $\tau$ such that, with positive probability and uniformly with respect to any initial positions $x$ and $y$ in $\X$,
the trajectories issued from $x$ intersect at time $\tau$ the trajectories issued from $y$ at some random times $s\in[0,\tau]$.
Even though they are not comparable in general, this \emph{path crossing} condition has connections with the notion of irreducibility of Markov processes, which means basically that for any $x$ and $y$ there is a deterministic time at which the trajectories issued from $x$ reach $y$ with positive probability~\cite{Norris}.

Consider for instance the semigroup $(M_{t})_{t\geq0}$ of a time continuous finite Markov chain $(X_t)_{t\geq 0}$:
$$
M_t f(x) = \mathbb{E}[f(X_t) \ | \ X_0=x].
$$
Irreducibility induces the existence of a random time $T(x,y)$ such that $X_{T(x,y)} = y$ and the strong Markov property gives
\begin{align*}
M_t f(x) \geq \mathbb{E}[f(X_t) \mathbf{1}_{T(x,y) \leq t} \ | \ X_0=x] 
= \mathbb{E}[M_{t-T(x,y)} f(y) \mathbf{1}_{T(x,y) \leq t}  \ | \ X_0= x]. 
\end{align*}
Assumption~\ref{H:irr} then holds with $\sigma_{x,y}$ being the law of $T(x,y)$ conditioned on being positive.

For Markov processes on infinite state spaces however, irreducibility does not imply~\ref{H:irr} in general, due to the requirement in~\ref{H:irr} for the probability of intersection and the crossing maximal time $\tau$ to be uniform in $x$ and~$y$.
Actually there are few examples of Markov processes on infinite state spaces which are irreducible since in most  cases the probability of reaching a point $y$ at a deterministic time starting from an arbitrary position $x$ is zero.
Similarly, the concept of irreducibility for positive semigroups, see~\cite[Definition C-III.3.1 p.306]{Nagel86} for instance, which proves very useful when working in spaces of functions, is rarely applicable in spaces of measures, exactly for the same reasons as for Markov processes.
Condition~\ref{H:irr} can then be seen as a useful alternative to irreducibility since it allows to cover many more interesting examples.

Consider for instance a Markov process on a compact topological state space $\X$.
If the random hitting time $T(x,y)$ has a positive density $\mathfrak{s}(x,y, \cdot)$ with respect to the Lebesgue measure which is continuous on its three components,
then the associated semigroup satisfies~\ref{H:irr} and~\ref{H:aper}.
A simple and typical example satisfying \ref{H:irr} but not \ref{H:aper} is given by the periodic (and irreducible) semigroup $(M_t)_{t\geq0}$ defined for every $f\in \B(\1)$, $t\geq 0$ and $x\in \X=[0,1]$ by
$$
M_tf(x) = f\left( x+t - \lfloor x+t \rfloor \right),
$$
where $\lfloor \cdot \rfloor$ stands for the integer part.
For this semigroup, \ref{H:irr} is satisfied only with $(\sigma_{x,y})_{0\leq x,y\leq1}$ a family of Dirac masses, so that~\ref{H:aper} does not hold.
A more sophisticated example (still irreducible) can be found in~\cite{BDG}, see also the comments in Section~\ref{sec:app}.

Condition~\ref{H:aper} is thus an aperiodicity assumption, reminiscent from the aperiodicity of Markov chains.
It can also be seen as a coupling condition since it means that, uniformly in $x$ and $x'$ in $\X$, the trajectories issued from $x$ and $x'$ intersect at time $\tau$ with positive probability (the intersection point being on a trajectory issued from some $y\in\X$ in the time interval $[0,\tau]$).

\

In Hypothesis~\ref{ass:global} the conditions are satisfied on the whole state space, which can be too much to ask in many applications, typically when the state space is not compact.
In that case Assumptions~\ref{H:irr} and~\ref{H:aper} can be localized through the use of Lyapunov functions, in the spirit of~\cite{BCGM,CV18}.
It leads to the following set of assumptions where, for two functions $f,g:\Omega\to\R$, the notation $f\lesssim g$ means that $f\leq C g$ on $\Omega$ for some $C>0$.

\begin{hyp}
\label{ass:local}
There exists a function $\psi:\X\to(0,\infty)$, a subset $K\subset\X$, a time $\tau>0$, and constants $\beta>\alpha>0$ and $\theta>0$ such that
\medskip
\begin{enumerate}[label=(A\arabic*), parsep=4mm, start = 0]
\item \label{A0} $\psi\leq V$ on $\X$ \quad \text{and} \quad $V\lesssim \psi  \ \ \text{on } K$; \quad $M V\lesssim V  \quad\text{and}\quad M \psi  \gtrsim \psi\quad \text{on}\ [0,\tau]\times\X,$
\item \label{A1} \makebox[.83\textwidth][c]{$M_\tau V \leq \alpha V + \theta  \mathbf{1}_K \psi$,}
\item \label{A2} \makebox[.83\textwidth][c]{$M_\tau \psi \geq \beta\psi$,}
\end{enumerate}
\bigskip
and there exist $c>0$ and a family of probability measures $(\sigma_{x,y})_{x,y\in K}$ over $[0,\tau]$ such that
\medskip
\begin{enumerate}[label=(H\arabic*'), parsep=4mm]
\item \label{H:irrloc} for all $f\in\B(V)$, $f\geq0$, and all $x,y\in K$,
\[\frac{M_{\tau} f(x)}{\psi(x)} \geq c \int_0^\tau \frac{M_{\tau-s} f(y)}{\psi(y)} \,\sigma_{x,y}(ds),\]
\item \label{H:aperloc}
\makebox[.83\textwidth][c]{
$\dis\sup_{x,x'\in K} \inf_{y \in K} \int_0^\tau\frac{M_{\tau-s}(\psi\1_K)(y)}{\psi(y)}(\sigma_{x,y}\wedge\sigma_{x',y})(ds)>0$.
}
\end{enumerate}
\end{hyp}

\

We are now ready to state our main result.

\begin{theo}
\label{th:main}
Assume that Hypothesis~\ref{ass:global} or Hypothesis~\ref{ass:local} is verified.
Then there exists a unique triplet $(\gamma,h,\lambda)\in\M_+( V)\times \B_+( V)\times \mathbb{R}$ of eigenelements of $(M_t)_{t\geq 0}$ with $\gamma (h) =\|h\|_{\B(V)}= 1$,
{\it i.e.} satisfying for all $t\geq0$
\begin{equation*}
\label{vecteursales}
\gamma M_t=\e^{\lambda t}\gamma\qquad\text{and}\qquad M_th=\e^{\lambda t}h.
\end{equation*}
Moreover,  there exist $C,\omega>0$ such that for all $t\geq0$ and  $\mu \in \M( V)$, 
\begin{equation}\label{eq:conv_norm}
\big\|\e^{-\lambda t}\mu M_t-\mu(h)\gamma\big\|_{\M( V)}\leq C\left\|\mu\right\|_{\M( V)}\e^{-\omega t}.
\end{equation}
\end{theo}

\

In~\cite{BCGM}, the authors give a set Assumption~{\bf A} of conditions which they prove to be equivalent to the convergence~\eqref{eq:conv_norm}.
These conditions are the same as Hypothesis~\ref{ass:local} where~\ref{H:irrloc} and~\ref{H:aperloc} are replaced by a Doeblin type coupling condition (A3) and an assumption (A4) which is basically a control on the ratio $M_t\psi(x)/M_t\psi(y)$ uniform in $x,y\in K$ and $t\geq0$.
The main advantage of Hypothesis~\ref{ass:local} compared to Assumption~{\bf A} in~\cite{BCGM} is that \ref{H:irrloc} is much easier to check than~(A4) on examples.
It will be illustrated on an example coming from evolutionary biology in Section~\ref{sec:app}.

The idea for proving Theorem~\ref{th:main} is to show that~\ref{H:irrloc} implies a slightly stronger control than~(A4) and, combined with~\ref{H:aperloc}, a slightly weaker version of~(A3).
Then we verify that the proof in~\cite{BCGM} can be adapted to these small variations of~(A3)-(A4).

Note that, contrary to Assumption~{\bf A} in~\cite{BCGM}, Hypothesis~\ref{ass:local} is sufficient but not necessary for ensuring the convergence~\eqref{eq:conv_norm},
see for instance the comments on the example in Section~\ref{sec:app}.
It is similar to the situation of finite Markov chains for which irreducibility and aperiodicity are a powerful sufficient but not necessary condition for proving convergence to an equilibrium.
Perron-Frobenius theory states that it becomes equivalent to the convergence only by adding the assumption that the Markov chain is undecomposable, in the sense that it has only one recurrent class.

\medskip

\begin{proof}[Proof of Theorem~\ref{th:main}]
We will prove that Hypothesis~\ref{ass:local} implies the existence of a constant $d>0$ such that for all $x,y$ in $K$ and all $t\geq0$
\begin{equation}\label{A4fort}
\frac{M_t\psi(x)}{\psi(x)}\geq d\,\frac{M_t\psi(y)}{\psi(y)},
\end{equation}
and the existence of a constant $\tilde c>0$ and a family $(\nu_{x,x'})_{x,x'\in K}$ of probability measures such that for all $x,x'$ in $K$ and all $y\in\{x,x'\}$
\begin{equation}\label{A3faible}
\frac{\delta_yM_\tau(\psi\,\cdot)}{M_\tau\psi(y)}\geq\tilde c\,\nu_{x,x'},
\end{equation}
and that Hypothesis~\ref{ass:global} implies these same inequalities with $K$ replaced by $\X$ and $\psi=\1$.
Such estimates, which first appeared in~\cite{CV14}, are a powerful tool to derive the exponential ergodicity of nonconservative semigroups, as attested by the further developments in~\cite{BCG17,BCGM,CV18}.

When~\eqref{A4fort} and~\eqref{A3faible} are verified on $\X$ with $\psi=\1$, and $t\mapsto\|M_t\1\|_\infty$ is locally bounded, we can use the proof of~\cite[Theorem~3.5]{BCG17} to get~\eqref{eq:conv_norm} with $V=\1$.
Indeed~\eqref{A4fort} readily implies~\cite[Assumption~(H2)]{BCG17} and~\eqref{A3faible} is similar to~\cite[Assumption~(H1)]{BCG17} with the difference that the coupling measure $\nu$ is independent of $x,x'$.
But it is not an issue since the proof in~\cite{BCG17} also works with a family $(\nu_{x,x'})$.
This was already noticed in~\cite{CV14}, where Assumption~(A1') is the analogue to~\eqref{A3faible}.

When~\eqref{A4fort} and~\eqref{A3faible} are verified in $K\subset\X$ and $\psi$ is as in Hypothesis~\ref{ass:local}, we can use the proof in~\cite{BCGM} to obtain~\eqref{eq:conv_norm}.
Similarly as above, \eqref{A4fort} implies~\cite[Assumption~(A4)]{BCGM} and~\eqref{A3faible} is the same as~\cite[Assumption~(A4)]{BCGM} except that in~\eqref{A3faible} the coupling measure depends on $x,x'$.
The proof in~\cite{BCGM} can be adapted to this slight variation by: replacing the quantity $\nu(M_{k\tau}\psi/\psi)$ by $\inf_K(M_{k\tau}\psi/\psi)$ in the definition of the Lyapunov functions $V_k$ and using accordingly~\eqref{A4fort} instead of~\cite[Assumption~(A4)]{BCGM}; letting $\nu$ depend on $x,x'$ in the proof of~\cite[Lemma~3.2]{BCGM} and using~\eqref{A3faible} instead of~\cite[Assumption~(A3)]{BCGM}

It thus only remains to check that Hypothesis~\ref{ass:local} implies~\eqref{A4fort} and~\eqref{A3faible}, and that Hypothesis~\ref{ass:global} implies the same inequalities with $K=\X$ and $\psi=\1$.
Since the proof for Hypothesis~\ref{ass:global} is simpler, we only give it for Hypothesis~\ref{ass:local} (replace below $K$ by $\X$ and $\psi$ by $\1$ for Hypothesis~\ref{ass:global}).

\medskip

By virtue of Assumption~\ref{A0}, the condition~\eqref{A4fort} is clearly verified for $t\in[0,\tau]$.
When $t>\tau$, using $f=M_{t-\tau}\psi$ in~\ref{H:irrloc} ensures that for all $x,y\in K$
\begin{equation}\label{eq:step1}
\frac{M_t\psi(x)}{\psi(x)} \geq c \int_0^\tau \frac{M_{t-s}\psi (y)}{\psi(y)} \sigma_{x,y}(ds).
\end{equation}
By Assumption~\ref{A0}, there exists a constant $C>0$ such that for all $y\in\X$, $t\geq0$ and $s\in[0,\tau]$
\begin{equation}\label{eq:step2}
\frac{M_t\psi(y)}{M_{t-s}\psi(y)}\leq C\frac{M_{k\tau}V(y)}{M_{k\tau}\psi(y)},
\end{equation}
where $k=\big\lfloor\frac{t-s}{\tau}\big\rfloor$.
Using Assumptions~\ref{A1} and~\ref{A2} we can prove as in~\cite[Lemma~4.1]{BCGM} that for all $y\in\X$
\begin{equation}\label{eq:step3}
\frac{M_{k\tau}V(y)}{M_{k\tau}\psi(y)}\leq\left(\frac{\alpha}{\beta}\right)^{\!k}\frac{V(y)}{\psi(y)}+\frac{\theta}{\beta-\alpha}.
\end{equation}
Since $\beta>\alpha$ and $V\lesssim\psi$ on $K$ by~\ref{A0}, we deduce from~\eqref{eq:step1}, \eqref{eq:step2} and~\eqref{eq:step3} that Condition~\eqref{A4fort} is verified for a constant $d>0$ independent of $x,y\in K$ and $t\geq0$.

We turn now to~\eqref{A3faible}.
Due to~\ref{A0}, there exists $C>0$ such that $M_\tau\psi\leq C\psi$ on $K$.
Hypothesis~\ref{H:irrloc} then yields that for every $x,y\in K$
\begin{equation}
\label{eq:irr-hom}
\frac{\delta_x M_{\tau} (\psi \,\cdot)}{M_\tau\psi(x)} \geq \frac{c}{C} \int_0^\tau \frac{\delta_y M_{\tau-s} (\psi\,\cdot)}{\psi(y)} \,\sigma_{x,y}(ds).
\end{equation}
By Hypothesis~\ref{H:aperloc} there exists $\epsilon>0$ such that for all $x,x'\in K$ we can find $y\in K$ such that
$$ 
m_{x,x',y}:=\int_0^\tau\frac{M_{\tau-s}(\psi\1_K)(y)}{\psi(y)}(\sigma_{x,y}\wedge\sigma_{x',y})(ds)\geq\epsilon.
$$
Setting
\[\nu_{x,x'}=\frac{1}{m_{x,x',y}}\int_0^\tau\frac{M_{\tau-s}(\psi\1_K\,\cdot)(y)}{\psi(y)}(\sigma_{x,y}\wedge\sigma_{x',y})(ds),\]
we deduce from~\eqref{eq:irr-hom} that~\eqref{A3faible} is satisfied with $\tilde c=c\epsilon/C$.

\end{proof}

\

\section{An application in population dynamics}
\label{sec:app}

We apply our method to the following nonlocal equation with drift:
\begin{equation}\label{eq:nonlocal+drift}
\partial_tu(t,x)+\partial_xu(t,x)=\int_\R u(t,y)Q(y,dx)\,dy+a(x)u(t,x),\qquad x\in\R.
\end{equation}
The solutions describe the density of traits in a population with nonlocal mutation in a space-time varying environment, encoding for instance the influence of a climate change through a spatial shift of the coefficient, see \cite{BDNZ,CH19} for motivations.
In the recent work~\cite{CH19}, Coville and Hamel address the Perron spectral problem associated to Equation~\eqref{eq:nonlocal+drift} in a slightly more general case, namely with a possibly non constant transport speed $q(x)$ and on a domain $\Omega$ which can be an open subinterval of $\R$.
In the present paper we choose to limit ourselves to the case $q\equiv1$ and $\Omega=\R$ in order to remain concise.
However, the method can be applied to other situations: to the price of additional technicalities for $q\not\equiv 1$ since the flow of the transport part may not be explicit,
or on the contrary with less calculations when $\Omega=(r_1,r_2)$ is a bounded interval ($-\infty<r_1<r_2<+\infty$) since no Lyapunov condition is needed in that case ({\it i.e.} Hypothesis~\ref{ass:global} is verified).

\

We assume that $a:\R\to\R$ is a continuous function such that

\begin{equation}
\label{eq:hypa}
\bar a=\sup_{x\in\R}a(x)<+\infty\qquad\text{and}\qquad \lim_{x\to\pm\infty}a(x)=-\infty
\end{equation}

\medskip

\noindent and $x\mapsto Q(x,\cdot)$ is a weak* continuous function $\R\to\M_+(\1)$ which satisfies

\begin{equation}
\label{eq:hypQmin}
\exists\,\epsilon,\kappa_0>0,\forall x\in\R,\quad Q(x,dy)\geq\kappa_0\1_{(x-\epsilon,x+\epsilon)}(y) dy
\end{equation}
\vskip-1.5mm
\begin{equation}
\label{eq:hypQmaj}
\bar Q=\sup_{x\in\R}Q(x,\R)<+\infty.
\end{equation}

\medskip

The assumptions on $Q$ are typically verified by a convolution kernel $Q(x,y)=J(x-y)$ with $J$ a finite positive measure such that $J(dz)\geq\kappa_0\1_{(-\epsilon,\epsilon)}(z)dz$.
They are less restrictive than in~\cite{CH19}, since we do not require the measures $Q(x,\cdot)$ to have a bounded and compactly supported Lebesgue density.
Our potential function $a$ satisfies a confining assumption $a(\pm\infty)=-\infty$ (no longer required if working in a bounded subinterval of $\R$), while it is supposed to be bounded in~\cite{CH19}.
The reason is that in~\cite{CH19}, only the existence of $(\lambda,h)$ is addressed, whereas we prove the existence of $(\lambda,h,\gamma)$ and the asynchronous exponential behavior of the solutions to the Cauchy problem.
The well-posedness of Equation~\eqref{eq:nonlocal+drift} completed with an initial condition $u(0,\cdot)=u_0$ in $\M(\1)$ or in the subspace $L^1(\R,dx)$ is a standard result.
Our method allows us to prove the following result stated in the $L^1$ framework, which is more usual in the partial differential equations community, but it is also valid for measure solutions by replacing the $L^1$ norm by the total variation norm.

\begin{theo}\label{th:application}
Under Assumptions~\eqref{eq:hypa},\eqref{eq:hypQmin} and \eqref{eq:hypQmaj}, there exist constants $C,\omega>0$ and a unique eigentriplet $(\lambda,\gamma,h)\in\R\times L^1\times C^1_b$ with $\gamma\geq0$, $h>0$ and $\int\!h\gamma  =\|h\|_\infty=1$, such that for any initial condition $u_0\in L^1(\R,dx)$ the corresponding solution $u(t,x)$
of Equation~\eqref{eq:nonlocal+drift} verifies
$$
\left\|u(t,\cdot) \e^{- \lambda t} - \Big(\int_\mathbb{R} h(x)u_0(x)\,dx\Big)\gamma \right\|_{L^1} \leq C \left\| u_0 \right\|_{L^1} \e^{-\omega t}.
$$
\end{theo}

\

The method of proof consists in applying our general result in Theorem~\ref{th:main} to the semigroup $(M_t)_{t\geq0}$ associated to Equation~\eqref{eq:nonlocal+drift}, which is defined through the Duhamel formula
\begin{equation}\label{eq:Duhamel}
M_tf(x)=f(x+t)\e^{\int_x^{x+t} a(s)ds}+\int_0^t \e^{\int_x^{x+s} a(s')ds'}\int_\R M_{t-s}f(y)Q(x+s,dy)\,ds.
\end{equation}
It is a standard result, see for instance~\cite{BCG17,BCGM} for more details on a similar model, that the Duhamel formula~\eqref{eq:Duhamel} indeed defines a positive semigroup $(M_t)_{t\geq0}$ in $\B(\1)$, but also in $C_b(\R)$ and in $C^1_b(\R)$.
When $f\in C^1_b(\R)$ it additionally holds that $(t,x)\mapsto M_tf(x)$ is continuously differentiable and satisfies
\[\partial_tM_tf=\L M_tf=M_t\L f,\]
where
\[\L f(x)=f'(x)+a(x)f(x)+\int_\R f(y)Q(x,dy).\]
This right action of the semigroup provides the unique solutions to the dual equation
\[\partial_t\varphi(t,x)=\partial_x\varphi(t,x)+a(x)\varphi(t,x)+\int_\R \varphi(t,y)Q(x,dy)\]
and thus the left action, defined by duality through $(\mu M_t)(f)=\mu(M_tf)$, yields the unique measure solutions to the direct equation~\eqref{eq:nonlocal+drift}.
The well-posedness of the direct equation in $L^1(\R,dx)$ ensures that if $\mu$ has a Lebesgue density, then so does $\mu M_t$ for all $t\geq0$.

\medskip

For applying Theorem~\ref{th:main}, we check that $(M_t)_{t\geq0}$ verifies Hypothesis~\ref{ass:local}.
Equation~\eqref{eq:nonlocal+drift} is a nice example to illustrate the conditions~\ref{H:irrloc} and~\ref{H:aperloc}, as we will explain now.
The drift term, which represents the spatial shift of the space-time varying environment, is crucial for verifying~\ref{H:irrloc}.
If we delete it, the situation is more involved: the result of Theorem~\ref{th:application} is still true when typically $1/(\bar a-a)$ is not locally integrable~\cite{AGK,Burger88,Cov10,LCW}, otherwise it may not hold due to the presence of a Dirac mass in $\gamma$, see~\cite{BurgerBomze,Cov13}.
Assumption~\eqref{eq:hypQmin} is crucial for verifying~\ref{H:aperloc}.
If we consider for instance the singular kernel $Q(x,\cdot)=\delta_{x-1}+\delta_{x+1}$, then the convergence in Theorem~\ref{th:application} does not hold and a periodic asymptotic behavior takes place, similarly as in~\cite{BDG,GM19}.

\medskip

Before proving Theorem~\ref{th:application}, with start with a useful strong positivity result about $M_t$.

\begin{lem}\label{lem:positivity}
Let $x_1,x_2\in\R$ with $x_1<x_2$. Then for any $y_1,y_2\in\R$ with $y_1<y_2$ and any $\tau>0$, there exists $\eta>0$ such that 
\[M_\tau\1_{[x_1,x_2]}\geq\eta\1_{[y_1,y_2]}.\]
\end{lem}

\begin{proof}[Proof of Lemma~\ref{lem:positivity}]
Iterating once the Duhamel formula~\eqref{eq:Duhamel} we get for all $t>0$
\begin{align*}
M_t\1_{[x_1,x_2]}(x)&\geq \e^{t\inf_{(x,x+t)}\! a}\int_0^t \int_\R \1_{[x_1,x_2]}(y+t-s)Q(x+s,dy)ds\\
&\geq \kappa_0t\,\e^{t\inf_{(x,x+t)}\! a} \int_{x+t-\epsilon}^{x+t+\epsilon} \1_{[x_1,x_2]}(z)dz
\end{align*}
from which we deduce that $M_t\1_{[x_1,x_2]}(x)\geq const>0$ for all $x\in[x_1-t-\epsilon/2,x_2-t+\epsilon/2]$.
Letting $n$ be a positive integer and considering $t=\tau/n$ we get that $M_t\1_{[x_1,x_2]}(x)\geq c_0>0$ for all $x\in[x_1-\tau/n-\epsilon/2,x_2-\tau/n+\epsilon/2]$,
and by iteration $M_\tau\1_{[x_1,x_2]}\geq c_0^n\1_{[x_1-\tau-n\epsilon/2,x_2-\tau+n\epsilon/2]}$.
The conclusion follows by choosing $n$ large enough.
\end{proof}

\begin{proof}[Proof of Theorem~\ref{th:application}]
We proceed in two steps.
First we check that on any bounded set $K$, Hypotheses~\ref{H:irrloc} and~\ref{H:aperloc} are satisfied for any $\tau>0$ (and any function $\psi$ bounded from above and below by a positive constant over $K$).
Then for Assumptions~\ref{A0}-\ref{A1}-\ref{A2}, we use $V=\1$ and build a suitable function $\psi$ such that the sublevel sets of $V/\psi$ are bounded.

\medskip

\paragraph{\bf Step \#1}

Let $K$ be a bounded set and $\tau\in(0,\epsilon/2)$.
We will build by induction two families $(\sigma_{x,y}^{t,n})$ and $(c_{x,y}^{t,n})$ indexed by $n\in\N$, $t\in[0,\tau]$, $x\in\R$ and $y\in[x+t-n\epsilon/2,x+t+n\epsilon/2]$ such that:
$\sigma_{x,y}^{t,n}$ is a probability measure on $[0,t]$ which has a positive Lebesgue density $\mathfrak s_{x,y}^{t,n}$ when $n\geq1$, $c_{x,y}^{t,n}$ is a positive constant, $(x,y,t,s)\mapsto \mathfrak s_{x,y}^{t,n}(s)$ and $(x,y,t)\mapsto c_{x,y}^{t,n}$ are continuous functions, and for all $f\geq0$
\begin{equation}\label{irr_n}
M_tf(x)\geq c_{x,y}^{t,n}\int_0^t M_{t-s}f(y)\,\sigma_{x,y}^{t,n}(ds).
\end{equation}
The integer $n$ essentially represents a number of jumps required for reaching $y$ by starting from~$x$.
Once these families are built, we choose $n$ large enough so that $K\subset [x+\tau-n\epsilon/2,x+\tau+n\epsilon/2]$ for every $x\in K$.
Positivity and continuity of $(x,y,t,s)\mapsto \mathfrak s_{x,y}^{t,n}(s)$ and $(x,y,t)\mapsto c_{x,y}^{t,n}$ then guarantee, together with Lemma~\ref{lem:positivity}, that Hypotheses~\ref{H:irrloc} and~\ref{H:aperloc} are satisfied with
\[\sigma_{x,y}=\sigma_{x,y}^{\tau,n}\qquad\text{and}\qquad c=\inf_{x,y\in K}c_{x,y}^{\tau,n}\]
for any function $\psi$ bounded from above and below by a positive constant on $K$.
Clearly if \ref{H:irrloc} and~\ref{H:aperloc} are satisfied for any $\tau\in(0,\epsilon/2)$, they are also satisfied for any $\tau>0$.
Let us now give the details of the induction.

\smallskip

For $n=0$, $y=x+t\geq x$ and the Duhamel formula~\eqref{eq:Duhamel} ensures that for any $f\geq0$
\begin{equation}\label{eq:n=0}
M_t f(x) \geq f(y) \,\e^{\int_x^{y} a(s) ds}.
\end{equation}
This transcribes that $y$ can be reached from $x$ by only following the drift, without doing any jump,
and it gives~\eqref{irr_n} with $\sigma_{x,y}^{t,0}=\delta_t$ and $c_{x,y}^{t,0}=\e^{\int_x^{y} a(s) ds}$.

\smallskip

For $n=1$, we start by noting that replacing $f$ by $M_{t' -(y-x)} f$ in~\eqref{eq:n=0} gives
\[M_{t'} f(x) \geq  \e^{\int_x^{y} a(s) ds} M_{t' -(y-x)} f(y)\]
for any $t'\geq y-x\geq0$.
Using this inequality with $t'=t-s$ and $x=z$ together with Assumption~\eqref{eq:hypQmin} in the Duhamel formula~\eqref{eq:Duhamel} we get for all $y\in [x+t-\epsilon/2, x+t+\epsilon/2]$, $t\in[0,\tau]$,
\begin{align*}
M_t f(x) 
&\geq \int_0^t \e^{\int_x^{x+s} a(s')ds'}\int_\R M_{t-s}f(z)Q(x+s,dz)\,ds \\
&\geq \kappa_0 \int_0^t \e^{\int_x^{x+s}\!a}\int_{x+s-\epsilon}^{x+s+\epsilon} M_{t-s}f(z)dz \,ds\\
&\geq \kappa_0 \int_0^t \e^{\int_x^{x+s}\!a}\int_{y+s-t}^{y} \e^{\int_z^y\!a}\, M_{t-s - (y-z)}f(y) dz \,ds\\
&\geq \kappa_0 \int_0^t M_{t-s'}f(y)\bigg(\int_0^{s'}\e^{\int_x^{x+s}\!a+\int_{y+s-s'}^y\!a}ds\bigg)ds'.
\end{align*}
This gives~\eqref{irr_n} when defining
\[c_{x,y}^{t,1}=\kappa_0 \int_0^t \int_0^{s}\e^{\int_x^{x+s'}\!a+\int_{y+s'-s}^y\!a}ds'ds>0\]
and $\sigma_{x,y}^{t,1}(ds)=\mathfrak s_{x,y}^{t,1}(s)ds$ with
\[\mathfrak s_{x,y}^{t,1}(s)=\frac{\kappa_0}{c_{x,y}^{t,1}} \int_0^s\e^{\int_x^{x+s'}\!a+\int_{y+s'-s}^y\!a}ds'>0.\]

\smallskip

For $n\to n+1$, the induction hypothesis ensures that for any $0\leq s\leq t\leq\tau$ and $z\in[y-t+s-n\epsilon/2,y-t+s+n\epsilon/2]$
\[M_{t-s}f(z)\geq c_{z,y}^{t-s,n}\int_s^t M_{t-s'}f(y)\,\mathfrak s_{z,y}^{t-s,n}(s'-s)ds'.\]
Injecting this inequality in the Duhamel formula~\eqref{eq:Duhamel} yields that for all $y\in[x+t-(n+1)\epsilon/2,x+t+(n+1)\epsilon/2]$, $t\in[0,\tau]$,
\begin{align*}
M_t f(x) 
&\geq \kappa_0 \int_0^t \e^{\int_x^{x+s}\!a}\int_{x+s-\epsilon}^{x+s+\epsilon} M_{t-s}f(z)dz \,ds\\
&\geq \kappa_0 \int_0^t \e^{\int_x^{x+s}\!a}\int_{\max(x+s-\epsilon,y-t+s-n\epsilon/2)}^{\min(x+s+\epsilon,y-t+s+n\epsilon/2)}  c_{z,y}^{t-s,n}\int_s^t M_{t-s'}f(y)\,\mathfrak s_{z,y}^{t-s,n}(s'-s)ds'dz \,ds\\
&\geq \kappa_0 \int_0^t M_{t-s'}f(y)\bigg(\int_0^{s'} \e^{\int_x^{x+s}\!a}\int_{\max(x+s-\epsilon,y-t+s-n\epsilon/2)}^{\min(x+s+\epsilon,y-t+s+n\epsilon/2)}  c_{z,y}^{t-s,n} \mathfrak s_{z,y}^{t-s,n}(s'-s)dz\,ds\bigg)ds',
\end{align*}
which gives~\eqref{irr_n} for $n+1$ with
\[c_{x,y}^{t,n+1}=\kappa_0\int_0^t \int_0^{s} \e^{\int_x^{x+s'}\!a}\int_{\max(x+s'-\epsilon,y-t+s'-n\epsilon/2)}^{\min(x+s'+\epsilon,y-t+s'+n\epsilon/2)}  c_{z,y}^{t-s',n} \mathfrak s_{z,y}^{t-s',n}(s-s')dzds'ds>0\]
and $\sigma_{x,y}^{t,n+1}(ds)=\mathfrak s_{x,y}^{t,n+1}(s)ds$ with
\[\mathfrak s_{x,y}^{t,n+1}(s)=\frac{\kappa_0}{c_{x,y}^{t,n+1}}\int_0^{s} \e^{\int_x^{x+s'}\!a}\int_{\max(x+s'-\epsilon,y-t+s'-n\epsilon/2)}^{\min(x+s'+\epsilon,y-t+s'+n\epsilon/2)}  c_{z,y}^{t-s',n} \mathfrak s_{z,y}^{t-s',n}(s-s')dzds'>0.\]

\medskip

\paragraph{\bf Step \#2}
Let us fix $\tau>0$.
We use drift conditions, similarly as in~\cite[Section 2.2]{BCGM}.
More precisely, we will prove that the functions
\[V=\1\qquad \text{and}\qquad \psi_0(x)=((1-(x/x_0)^2)_+)^2\]
satisfy
\[\L V\leq\alpha_0 V+\theta_0\psi_0\qquad\text{and}\qquad\L\psi_0\geq\beta_0\psi_0\]
for some $x_0\geq\epsilon$ and some constants $\alpha_0=\beta_0-1$ and $\theta_0>0$.
From these drift conditions, we get that the function $\phi=V-\theta_0\psi_0$ verifies
\[\L\phi\leq\alpha_0 V+\theta_0\psi_0-\theta_0\beta_0\psi_0=\alpha_0\phi\]
and, since $\partial_tM_t\phi=M_t\L\phi$, we deduce from Grönwall's lemma that $M_\tau\phi\leq\e^{\alpha_0\tau}\phi$, which gives
\[M_\tau V\leq \e^{\alpha_0\tau}\left(V-\theta_0 \psi_0\right)+ \theta_0 M_\tau\psi_0\leq\e^{\alpha_0\tau}V+ \theta_0 M_\tau\psi_0.\]
We then prove the existence of $\zeta>0$ such that $M_\tau\psi_0\leq\zeta V$ and we set $\psi=\zeta^{-1} M_\tau\psi_0$, so that $\psi>0$ due to Lemma~\ref{lem:positivity}, $\psi\leq V$, and
\begin{equation}\label{eq:presqueA1}M_\tau V\leq\e^{\alpha_0\tau}V+ \theta_0 \zeta\psi.
\end{equation}
Besides, $\L\psi_0\geq\beta_0\psi_0$ yields $M_{2\tau}\psi_0\geq \e^{\beta_0\tau}M_\tau\psi_0$, which also reads
\[M_\tau\psi\geq\e^{\beta_0\tau}\psi.\]
Choosing
\[R>\frac{\theta_0\zeta}{ \e^{\alpha_0\tau}(\e^\tau-1)}\qquad\text{and}\qquad K=\{V\leq R\psi\}\]
we obtain from~\eqref{eq:presqueA1} that
\[M_\tau V\leq\Big(\e^{\alpha_0\tau}+\frac{\theta_0\zeta}{R }\Big)V+ \theta_0\zeta\1_K\psi\]
and Assumptions~\ref{A1} and~\ref{A2} are verified with $\theta=\theta_0\zeta$ and
\[\alpha=\e^{\alpha_0\tau}+\frac{\theta_0\zeta}{R}<\e^{(\alpha_0+1)\tau}=\e^{\beta_0\tau}=\beta.\]
Finally we check that $\{V\leq R\psi\}$ is bounded, which ends the verification of Assumptions~\ref{A0}-\ref{A1}-\ref{A2}.

\smallskip

Now we show that $\L\psi_0\geq\beta_0\psi_0$ for some $\beta_0$. We have
\begin{align*}
\L\psi_0(x)&=-4\frac{x}{x_0^2}(1-(x/x_0)^2)_++a(x)\psi_0(x)+\int_{-x_0}^{x_0}{(1-(y/x_0)^2)}^2Q(x,dy)\\
&\geq -\frac{4}{x_0}(1-(x/x_0)^2)_++\Big[\inf_{(-x_0,x_0)}\!a\Big] \psi_0(x)+\kappa_0x_0\int_{-1}^{1}{(1-y^2)}^2\1_{(x-\epsilon,x+\epsilon)}(x_0y)\, dy\\
&\geq -\frac{4}{x_0}(1-(x/x_0)^2)_++\Big[\inf_{(-x_0,x_0)}\!a\Big] \psi_0(x)+\kappa_0\1_{(-x_0,x_0)}(x)x_0\int_{1-\epsilon/x_0}^1(1-y)^2(1+y)^2\,dy\\
&\geq -\frac{4}{x_0}(1-(x/x_0)^2)_++\Big[\inf_{(-x_0,x_0)}\!a\Big] \psi_0(x)+\frac{8\kappa_0\epsilon^3}{15\,x_0^2}\1_{(-x_0,x_0)}(x),
\end{align*}
since for $r\in[0,1]$ we have $\int_{1-r}^1(1-y)^2(1+y)^2\,dy=\int_0^r z^2(2-z)^2dz=r^3\big(\frac43-r+\frac{r^2}{5}\big)\geq\frac{8r^3}{15}$.
If $2\kappa_0\epsilon^3\geq15 x_0$ we get
\[\L\psi_0\geq\Big[\inf_{(-x_0,x_0)}\!a\Big] \psi_0.\]
Otherwise we split, for $|x|\geq\sqrt{1-(2\kappa_0\epsilon^3)/(15x_0)}\,x_0$,
\[\L\psi_0(x)\geq\Big[\inf_{(-x_0,x_0)}\!a\Big] \psi_0(x),\]
and for $|x|\leq\sqrt{1-(2\kappa_0\epsilon^3)/(15x_0)}\,x_0$, we have $\sqrt{\psi_0(x)}\geq\frac{2\kappa_0\epsilon^3}{15x_0}$ which yields $\sqrt{\psi_0(x)}\leq\frac{15x_0}{2\kappa_0\epsilon^3}\psi_0(x)$ and thus
\[\L\psi_0(x)\geq\bigg[-\frac{30}{\kappa_0\epsilon^3}+\inf_{(-x_0,x_0)}\!a+\frac{8\kappa_0\epsilon^3}{15\,x_0^2}\bigg]\psi_0(x).\]
At the end, in any case and for any $x_0\geq\epsilon$,
\[\L\psi_0\geq\bigg[-\frac{30}{\kappa_0\epsilon^3}+\inf_{(-x_0,x_0)}\!a\bigg]\psi_0:=\beta_0\psi_0.\]

\smallskip

Next, for $V=\1$, we have
\[\L V(x)=a(x)+Q(x,\R)\leq a(x)+\overline Q.\]
Choosing $r_0>0$ such that $a(x)<-\bar Q+\beta_0-1:=-\bar Q+\alpha_0$ when $|x|>r_0$ we get
\[\L V(x)\leq\alpha_0V(x)+(\bar a+\bar Q)\1_{(-r_0,r_0)}(x).\]
So if we choose $x_0\geq\sqrt2r_0$ in the definition of $\psi_0$ we get $\1_{(-r_0,r_0)}\leq4\psi_0$ and thus
\[\L V\leq\alpha_0V+4(\bar a+\bar Q)\psi_0:=\alpha_0V+\theta_0\psi_0.\]
Besides, since $\L V\leq\bar a+\bar Q$, we have by Grönwall's lemma that
\[M_\tau\psi_0\leq M_\tau V\leq \e^{-(\bar a+\bar Q)\tau}V:=\zeta V.\]

\smallskip

Finally $K=\{V\leq R\psi\}$ is bounded because $\psi$ tends to zero at $\pm\infty$: from the Duhamel formula~\eqref{eq:Duhamel} we have
\[M_\tau\psi_0(x)\leq\psi_0(x+t)\e^{\int_x^{x+\tau} a(s)ds}+\bar Q\int_0^\tau \e^{\int_x^{x+s} a(s')ds'}\e^{(t-s)(\bar a+\bar Q)}ds\]
which yields the result since $\lim_{x\to\pm\infty}a(x)=-\infty$.

\medskip

\paragraph{\bf Conclusion}
We have checked Hypothesis~\ref{ass:local} so the result is a direct application of Theorem~\ref{th:main}, the regularity of $\gamma$ following from the convergence in total variation norm and the invariance of $L^1$ under the semigroup, and the regularity of $h$ being obtained by usual arguments from Equation~\eqref{eq:Duhamel}.
\end{proof}

\

\section*{Acknowledgments}

The authors are very grateful to Vincent Bansaye and Aline Marguet for useful discussions about the results of the present note.

B.C. has received the support of the Chair ``Mod\'elisation Math\'ematique et Biodiversit\'e'' of VEOLIA-Ecole Polytechni\-que-MnHn-FX
and the ANR project MESA (ANR-18-CE40-006), funded by the French Ministry of Research.


\end{document}